\documentclass{amsart}
\usepackage{amssymb}
\usepackage{amsmath}
\usepackage{amsfonts}
\usepackage{geometry}

\newtheorem{theorem}{Theorem}
\theoremstyle{plain}
\newtheorem{acknowledgement}{Acknowledgement}

\newtheorem{condition}{Condition}

\newtheorem{definition}{Definition}
\newtheorem{example}{Example}

\newtheorem{remark}{Remark}

\numberwithin{equation}{section}
\input{tcilatex}

\begin{document}
\title{Doob--Meyer for rough paths}
\author{Peter Friz, Atul Shekhar}
\address{TU\ and WIAS\ Berlin (first and corresponding author,
friz@math.tu-berlin.de), TU\ Berlin (second author)}

\begin{abstract}
Recently, Hairer--Pillai proposed the notion of $\theta $-roughness of a
path which leads to a deterministic Norris lemma. In the Gubinelli framework
(H\"{o}lder, level $2$) of rough paths, they were then able to prove a H\"{o}%
rmander type result (SDEs driven by fractional Brownian motion, $H>1/3$). We
take a step back and propose a natural "roughness" condition relative to a
given $p$-rough path in the sense of Lyons; the aim being a Doob-Meyer
result for rough integrals in the sense of Lyons. The interest in our
(weaker) condition is that it is immediately verified for large classes of
Gaussian processes, also in infinite dimensions. We conclude with an
application to non-Markovian system under H\"{o}rmander's condition.
\end{abstract}

\maketitle

\section{Introduction}

Recently, Hairer--Pillai \cite{HP11} proposed the notion of $\theta $%
-roughness of a path which leads to a deterministic Norris lemma, i.e. some
sort of quantitative Doob-Meyer decomposition, for (level-2, H\"{o}lder)
rough integrals in the sense of Gubinelli. It is possible to check that this
roughness condition holds for fractional Brownian motion (fBm); indeed in 
\cite{HP11} the author show $\theta $-roughness for any $\theta >H$ where $H$
denotes the Hurst parameter. (Recall that Brownian motion corresponds to $%
H=1/2$; in comparison, the regime $H<1/2$ should be thought of as
"rougher".) All this turns out to be a key ingredient in their H\"{o}rmander
type result for stochastic differential equations driven by fBm, any $H>1/3$%
, solutions of which are in general non-Markovian.

In the present note we take a step back and propose a natural "roughness"
condition relative to a given $p$-rough path (of arbitrary level $\left[ p%
\right] =1,2,\dots $) in the sense of Lyons; the aim being a Doob-Meyer
result for (general) rough integrals in the sense of Lyons. The interest in
our (weaker) condition is that it is immediately verified for large classes
of Gaussian processes, also in infinite dimensions. (In essence one only
needs a Khintchine law of iterated logarithms for $1$-dimensional
projections.)

We conclude with an application to non-Markovian systems under H\"{o}%
rmander's condition, in the spirit of \cite{CF}.

\section{Truely "rough" paths and a deterministic Doob-Meyer result}

Let $V$ be a Banach-space. Let $p\geq 1$. Assume $f\in Lip^{\gamma }\left(
V,L\left( V,W\right) \right) $, $\gamma >p-1$, and $\mathbf{X:}\left[ 0,T%
\right] \rightarrow V$ to be a $p$-rough path in the sense of T. Lyons \cite%
{L98, LQ02} controlled by $\omega $.

Recall that such a rough path consists of a underlying path $X:\left[ 0,T%
\right] \rightarrow V$, together with higher order information which
somewhat \textit{prescribes} the iterated integrals $\int_{0}^{\cdot
}dX_{t_{1}}\otimes ...\otimes dX_{t_{k}}$ for $1<k\leq \left[ p\right] $.

\begin{definition}
\label{DefRough}For fixed $s\in \lbrack 0,T)$ we call $X$ "\textit{rough at
time }$s$" if (convention $0/0:=0$)%
\begin{equation*}
\left( \ast \right) :\forall v^{\ast }\in V^{\ast }\backslash \left\{
0\right\} :\lim \sup_{t\downarrow s}\frac{\left\vert \left\langle v^{\ast
},X_{s,t}\right\rangle \right\vert }{\omega \,\left( s,t\right) ^{2/p}}%
=+\infty .
\end{equation*}%
If $X$ is rough on some dense set of $\left[ 0,T\right] $, we call it 
\textit{truely rough}.
\end{definition}

\begin{theorem}
\label{ThmDBRP}(i) Assume $X$ is rough at time $s$. Then 
\begin{equation*}
\int_{s}^{t}f\left( X\right) d\mathbf{X}=O\left( \omega \,\left( s,t\right)
^{2/p}\right) \text{ as }t\downarrow s\implies f\left( X_{s}\right) =0\text{.%
}
\end{equation*}%
(i') As a consequence, if $X$ is truely rough, then%
\begin{equation*}
\int_{0}^{\cdot }f\left( X\right) d\mathbf{X}\equiv 0\text{ on }\left[ 0,T%
\right] \implies f\left( X_{\cdot }\right) \equiv 0\text{ on }\left[ 0,T%
\right] .
\end{equation*}%
(i\textquotedblright ) As another consequence, assume $g\in C\left(
V,W\right) $ and $\left\vert t-s\right\vert =O(\omega \,\left( s,t\right)
^{2/p})$, satisfied e.g. when $\omega \left( s,t\right) \asymp t-s$ and $%
p\geq 2$ (the "rough" regime of usual interest) then 
\begin{equation*}
\int_{0}^{\cdot }f\left( X\right) d\mathbf{X}+\int_{0}^{\cdot }g\left(
X\right) dt\equiv 0\text{ on }\left[ 0,T\right] \implies f\left( X_{\cdot
}\right) ,g\left( X_{\cdot }\right) \equiv 0\text{ on }\left[ 0,T\right] .
\end{equation*}%
(ii) Assume $Z:=X\oplus Y$ lifts to a rough path and set, with $\tilde{f}%
\left( z\right) \left( x,y\right) :=f\left( z\right) x$, 
\begin{equation*}
\int f\left( Z\right) d\mathbf{X:=}\int \tilde{f}\left( Z\right) d\mathbf{Z}%
\text{.}
\end{equation*}%
Then the conclusions from (i),(i') and (i\textquotedblright ), with $%
g=g\left( Z\right) $, remain valid.
\end{theorem}

\begin{remark}
Solutions of rough differential equations $dY=V\left( Y\right) d\mathbf{X}$
in the sense of Lyons are understood in the integral sense, based on the
integral defined in (ii) above. This is our interest in this (immediate)
extension of part (i).
\end{remark}

\begin{proof}
(i) A basic estimate (e.g. \cite{CQ02}) for the $W$-valued rough integral is%
\begin{equation*}
\int_{s}^{t}f\left( X\right) dX=f\left( X_{s}\right) X_{s,t}+O\left( \omega
\,\left( s,t\right) ^{2/p}\right) .
\end{equation*}%
By assumption, for fixed $s\in \lbrack 0,T)$, we have%
\begin{equation*}
0=\frac{f\left( X_{s}\right) X_{s,t}}{\omega \,\left( s,t\right) ^{2/p}}%
+O\left( 1\right) \text{ as }t\downarrow s
\end{equation*}%
and thus, for any $w^{\ast }\in W^{\ast }$,%
\begin{equation*}
\,\frac{|\left\langle v^{\ast },X_{s,t}\right\rangle |}{\omega \,\left(
s,t\right) ^{2/p}}:=\left\vert \left\langle w^{\ast },\frac{f\left(
X_{s}\right) X_{s,t}}{\omega \,\left( s,t\right) ^{2/p}}\right\rangle
\right\vert =O\left( 1\right) \text{ as }t\downarrow s;
\end{equation*}%
where $v^{\ast }\in V^{\ast }$ is given by $V\ni v\mapsto \left\langle
w^{\ast },f\left( X_{s}\right) v\right\rangle $ recalling that $f\left(
X_{s}\right) \in L\left( V,W\right) $. Unless $v^{\ast }=0$, the assumption $%
(\ast )$ implies that, along some sequence $t_{n}\downarrow s$, we have the
divergent behaviour $\left\vert \left\langle v^{\ast
},X_{s,t_{n}}\right\rangle \right\vert /\omega \,\left( s,t_{n}\right)
^{2/p}\rightarrow \infty $, which contradicts that the same expression is $%
O\left( 1\right) $ as $t_{n}\downarrow s$. We thus conclude that $v^{\ast
}=0 $. In other words,%
\begin{equation*}
\forall w^{\ast }\in W^{\ast },v\in V:\left\langle w^{\ast },f\left(
X_{s}\right) v\right\rangle =0.
\end{equation*}%
and this clearly implies $f\left( X_{s}\right) =0$. (Indeed, assume
otherwise i.e. $\exists v:w:=f\left( X_{s}\right) v\neq 0$. Then define $%
\,\left\langle w^{\ast },\lambda w\right\rangle :=\lambda $ and extend,
using Hahn-Banach if necessary, $w^{\ast }$ from \textrm{span}$\left(
w\right) \subset W$ to the entire space, such as to obtain the contradiction 
$\left\langle w^{\ast },f\left( X_{s}\right) v\right\rangle =1$.)$\newline
$\newline
(i\textquotedblright ) From the assumptions, $\int_{s}^{t}g\left(
X_{r}\right) dr\leq \left\vert g\right\vert _{\infty }\left\vert
t-s\right\vert =O\left( \omega \,\left( s,t\right) ^{2/p}\right) .$ We may
thus use (i) to conclude $f\left( X_{s}\right) =0$ on $s\in \lbrack 0,T)$.
It follows that $\int_{0}^{\cdot }g\left( X_{r}\right) dr\equiv 0$ and by
differentiation, $g\left( X_{\cdot }\right) \equiv 0$ on $\left[ 0,T\right] $%
.\newline
(ii) By definition of $\int f\left( Z\right) d\mathbf{X}$ and $\tilde{f}$
respectively,%
\begin{eqnarray*}
\int_{s}^{t}f\left( Z\right) d\mathbf{X}\mathbf{:=} &&\int_{s}^{t}\tilde{f}%
\left( Z\right) d\mathbf{Z} \\
&\mathbf{=}&\tilde{f}\left( Z_{s}\right) Z_{s,t}+O\left( \omega \,\left(
s,t\right) ^{2/p}\right) \\
&=&f\left( Z_{s}\right) X_{s,t}+O\left( \omega \,\left( s,t\right)
^{2/p}\right)
\end{eqnarray*}%
and the identical proof (for (i'), then (i\textquotedblright )) goes
through, concluding $f\left( Z_{s}\right) =0$.
\end{proof}

\begin{remark}
The reader may wonder about the restriction to $p\geq 2$ in
(i\textquotedblright ) for H\"{o}lder type controls $\omega \left(
s,t\right) \asymp t-s$. Typically, when $p<2$, one uses Young theory,
thereby avoiding the full body of rough path theory. That said, one can
always view a path of finite $p$-variation, $p<2$, as rough path of finite $%
2 $-variation (iterated integrals are well-defined as Young integrals).
Moreover, by a basic consistency result, the respective integrals (Young,
rough) coincide. In the context of fBM with Hurst parameter $H\in \left(
1/2,1\right) ,$ for instance, we can take $p=2$ and note that in this
setting fBM is truely rough (cf. below for a general argument based on the
law of iterated logarithm). By the afore-mentioned consistency, the
Doob--Meyer decomposition of (i\textquotedblright ) then becomes a statement
about Young integrals. Such a decomposition was previously used in \cite{BH}.
\end{remark}

\begin{remark}
The argument is immediately adapted to the Gubinelli setting of "controlled"
paths and would (in that context) yield uniqueness of the derivative process.
\end{remark}

\begin{remark}
In definition \ref{DefRough}, one could replace the denominator $\omega
\,\left( s,t\right) ^{2/p}$ by $\omega \left( s,t\right) ^{\theta }$, say
for $1/p<$ $\theta \leq 2/p$. Unlike \cite{HP11}, where $2/p-\theta $
affects the quantitative estimates, there seems to be no benefit of such a
stronger condition in the present context.
\end{remark}

\section{True roughness of stochastic processes}

Fix $\rho \in \lbrack 1,2)$ and $p\in \left( 2\rho ,4\right) $. We assume
that the $V$-valued stochastic process $X$ lifts to a random $p$-rough path.
We assume $V^{\ast }$ separable which implies separability of the unit
sphere in $V^{\ast }$ and also (by a standard theorem) separability of $V$.
(Separability of the dual unit sphere in the weak-$\ast $ topology,
guaranteed when $V$ is assumed to be separable, seems not enough for our
argument below.)

The following $2$ conditions should be thought of as a weak form of a LIL
lower bound, and a fairely robust form of a LIL upper bound. As will be
explained below, they are easily checked for large classes of Gaussian
processes, also in infinite dimensions.

\begin{condition}
Set $\psi \left( h\right) =h^{\frac{1}{2\rho }}\left( \ln \ln 1/h\right)
^{1/2}$. Assume (i) there exists $c>0$ such that for every fixed dual unit
vector $\varphi \in V^{\ast }$ and $s\in \lbrack 0,T)$ 
\begin{equation*}
\text{ }\mathbb{P}\left[ \lim \sup_{t\downarrow s}\left\vert \varphi \left(
X_{s,t}\right) \right\vert /\psi \left( t-s\right) \geq c\right] =1
\end{equation*}%
and (ii) for every fixed $s\in \lbrack 0,T)$, 
\begin{equation*}
\mathbb{P}\left[ \lim \sup_{t\downarrow s}\frac{\left\vert
X_{s,t}\right\vert _{V}}{\psi \left( t-s\right) }<\infty \right] =1
\end{equation*}
\end{condition}

\begin{theorem}
\label{ThmGrough}\bigskip Assume $X$ satisfies the above condition. Then $X$
is a.s. truely rough.
\end{theorem}

\begin{proof}
Take a dense, countable set of dual unit vectors, say $K\subset V^{\ast }$.
Since $K$ is countable, the set on which condition (i) holds simultanously
for all $\varphi \in K$ has full measure,%
\begin{equation*}
\mathbb{P}\left[ \forall \varphi \in K:\lim \sup_{t\downarrow s}\left\vert
\varphi \left( X_{s,t}\right) \right\vert /\psi \left( t-s\right) \geq c%
\right] =1
\end{equation*}%
On the other hand, every unit dual vector $\varphi \in V^{\ast }$ is the
limit of some $\left( \varphi _{n}\right) \subset K$. Then%
\begin{equation*}
\frac{\left\vert \left\langle \varphi _{n},X_{s,t}\right\rangle \right\vert 
}{\psi \left( t-s\right) }\leq \frac{\left\vert \left\langle \varphi
,X_{s,t}\right\rangle \right\vert }{\psi \left( t-s\right) }+\left\vert
\varphi _{n}-\varphi \right\vert _{V^{\ast }}\frac{\left\vert
X_{s,t}\right\vert _{V}}{\psi \left( t-s\right) }
\end{equation*}%
so that, using $\overline{\lim }\left( \left\vert a\right\vert +\left\vert
b\right\vert \right) \leq \overline{\lim }\left( \left\vert a\right\vert
\right) +\overline{\lim }\left( \left\vert b\right\vert \right) $, and
restricting to the above set of full measure,%
\begin{equation*}
c\leq \overline{\lim_{t\downarrow s}}\frac{\left\vert \left\langle \varphi
_{n},X_{s,t}\right\rangle \right\vert }{\psi \left( t-s\right) }\leq 
\overline{\lim_{t\downarrow s}}\frac{\left\vert \left\langle \varphi
,X_{s,t}\right\rangle \right\vert }{\psi \left( t-s\right) }+\left\vert
\varphi _{n}-\varphi \right\vert _{V^{\ast }}\overline{\lim_{t\downarrow s}}%
\frac{\left\vert X_{s,t}\right\vert _{V}}{\psi \left( t-s\right) }.
\end{equation*}%
Sending $n\rightarrow \infty $ gives, with probability one, 
\begin{equation*}
c\leq \overline{\lim_{t\downarrow s}}\frac{\left\vert \left\langle \varphi
,X_{s,t}\right\rangle \right\vert }{\psi \left( t-s\right) }.
\end{equation*}%
Hence, for a.e. sampe $X=X\left( \omega \right) $ we can pick a sequence $%
\left( t_{n}\right) $ converging to $s$ such that $\left\vert \left\langle
\varphi ,X_{s,t_{n}}\right\rangle \right\vert /\psi \left( t_{n}-s\right)
\geq c-1/n$. On the other hand, for any $\theta \geq 1/\left( 2\rho \right) $
\begin{eqnarray*}
\frac{\left\vert \left\langle \varphi ,X_{s,t_{n}}\left( \omega \right)
\right\rangle \right\vert }{\left\vert t_{n}-s\right\vert ^{\theta }} &=&%
\frac{\left\vert \left\langle \varphi ,X_{s,t_{n}}\left( \omega \right)
\right\rangle \right\vert }{\psi \left( t_{n}-s\right) }\frac{\psi \left(
t_{n}-s\right) }{\left\vert t_{n}-s\right\vert ^{\theta }} \\
&\geq &\left( c-1/n\right) \left\vert t_{n}-s\right\vert ^{\frac{1}{2\rho }%
-\theta }L\left( t_{n}-s\right) \\
&\rightarrow &\infty \text{ \ \ }
\end{eqnarray*}%
since $c>0$ and $\theta \geq 1/\left( 2\rho \right) $ and slowly varying $%
L\left( \tau \right) :=\left( \ln \ln 1/\tau \right) ^{1/2}$ (in the extreme
case $\theta =1/\left( 2\rho \right) $ the divergence is due to the (very
slow) divergence $L\left( \tau \right) \rightarrow \infty $ as $\tau
=t_{n}-s\rightarrow 0$ .)
\end{proof}

\subsection{Gaussian processes}

The conditions put forward here are typical for Gaussian process\ (so that
the pairing $\left\langle \varphi ,X\right\rangle $ is automatically a
scalar Gaussian process). Sufficient conditions for (i), in fact, a law of
iterated logarithm, with equality and $c=1$ are e.g. found in \cite[Thm
7.2.15]{MR}. These conditions cover immediately - and from general
principles - many Gaussian (rough paths) examples, including fractional
Brownian motion ($\rho =1/\left( 2H\right) $, lifted to a rough path \cite%
{CQ02, FV10}) and the stationary solution to the stochastic heat equation on
the torus, viewed as as Gaussian processes parametrized by $x\in \lbrack
0,2\pi ]$; here $\rho =1$, the fruitful lift to a "spatial" Gaussian rough
path is due to Hairer \cite{H11}.

As for condition (ii), it holds under a very general condition \cite[Thm A.22%
]{FV10}%
\begin{equation*}
\exists \eta >0:\sup_{0\leq s,t\leq T}E\exp \left( \eta \frac{\left\vert
X_{s,t}\right\vert _{V}^{2}}{\left\vert t-s\right\vert ^{1/\rho }}\right)
<\infty .
\end{equation*}%
In presence of some scaling, this condition is immediately verfied by
Fernique's theorem.

\begin{example}
$d$-dimensional fBM is a.s. truely rough (in fact, $H$-rough)
\end{example}

In order to apply this in the context of (random) rough integration, we need
to intersect the class of truely rough Gaussian processes with the classes
of Gaussian processes which amit a rough path lift. To this end, we recall
the following standard setup \cite{FV10}. Consider a continuous $d$%
-dimensional Gaussian process, say $X$, realized as coordinate process on
the (not-too abstract) Wiener space $\left( E,\mathcal{H},\mu \right) $
where $E=C\left( \left[ 0,T\right] ,\mathbb{R}^{d}\right) $ equipped with $%
\mu $ is a Gaussian measure s.t. $X$ has zero-mean, independent components
and that $V_{\rho \text{-var}}\left( R,\left[ 0,T\right] ^{2}\right) $, the $%
\rho $-variation in 2D sense of the covariance $R$ of $X$, is finite for $%
\rho \in \lbrack 1,2)$. (In the fBM case, this condition translates to $%
H>1/4 $). From \cite[Theorem 15.33]{FV10} it follows that we can lift the
sample paths of $X$ to $p$-rough paths for any $p>2\rho $ and we denote this
process by $\mathbf{X}$, called the \emph{enhanced Gaussian process}. In
this context, modulo a deterministic time-change, condition (ii) will always
be satisfied (with the same $\rho $). The non-degeneracy condition (i), of
course, cannot be expect to hold true in this generality; but, as already
noted, conditions are readily available \cite{MR}.

\begin{example}
$Q$-Wiener processes are a.s. truely rough. More precisely, consider a
separable Hilbert space $H$ with ONB $\left( e_{k}\right) $, $\left( \lambda
_{k}\right) \in l^{1}$, $\lambda _{k}>0$ for all $k$, and a countable
sequence $\left( \beta ^{k}\right) $ of independent standard Brownians. Then
the limit%
\begin{equation*}
X_{t}:=\sum_{k=1}^{\infty }\lambda _{k}^{1/2}\beta _{t}^{k}e_{k}
\end{equation*}%
exists a.s. and in $L^{2}$, uniformly on compacts and defines a $Q$-Wiener
process, where $Q=\sum \lambda _{k}\,\left\langle e_{k},\cdot \right\rangle $
is symmetric, non-negative and trace-class. (Conversely, any such operator $%
Q $ on $H$ can be written in this form and thus gives rise to a $Q$-Wiener
process.) By Brownian scaling and Fernique, condition (ii) is obvious. As
for condition (i), let $\varphi $ be an arbitrary unit dual vector and note
that $\varphi \left( X_{\cdot }\right) /\sigma _{\varphi }$ is standard\
Brownian provided we set%
\begin{equation*}
\sigma _{\varphi }^{2}:=\sum \lambda _{k}\left\langle \varphi
,e_{k}\right\rangle ^{2}>0.
\end{equation*}%
By Khintchine's law of iterated logarithms for standard Brownian motion, for
fixed $\varphi $ and $s$, with probability one,%
\begin{equation*}
\lim \sup_{t\downarrow s}\left\vert \varphi \left( X_{s,t}\right)
\right\vert /\psi \left( t-s\right) \geq \sqrt{2}\sigma _{\varphi }.
\end{equation*}%
Since $\varphi \mapsto \sigma _{\varphi }^{2}$ is weakly continuous (this
follows from $\left( \lambda \right) \in l^{1}$ and dominated convergence)
and compactness of the unit sphere in the weak topology, $c:=\inf \sigma
_{\varphi }>0$, and so condition (ii) is verified.
\end{example}

Let us quickly note that $Q$-Wiener processes can be naturally enhanced to
rough paths. Indeed, it suffices to define the $H\otimes H$-valued "second
level" increments as%
\begin{equation*}
\left( s,t\right) \mapsto \mathbb{X}_{s,t}:=\sum_{i,j}\lambda
_{i}^{1/2}\lambda _{j}^{1/2}\int_{s}^{t}\beta _{s,\cdot }^{i}\circ d\beta
^{j}e_{i}\otimes e_{j}.
\end{equation*}%
which essentially reduces the construction of the "area-process" to the L%
\'{e}vy area of a 2-dimensional standard Brownian motion. (Alternatively,
one could use integration against $Q$-Wiener processes.) Rough path
regularity, $\left\vert \mathbb{X}_{s,t}\right\vert _{H\otimes H}=O\left(
\left\vert t-s\right\vert ^{2\alpha }\right) $ for some $\alpha \in
(1/3,1/2] $ (in fact: any $\alpha <1/2$), is immediate from a suitable
Kolmogrov-type or GRR criterion (e.g. \cite{FH12, FV10}).

Variations of the scheme are possible of course, it is rather immediate to
define $Q$-Gaussian processes in which $\left( \beta ^{k}\right) $ are
replaced by $\left( X^{k}\right) $, a sequence of independent Gaussian
processes, continuous each with covariance uniformly of finite $\rho $%
-variation, $\rho <2$.

Let us insist that the (random) rough integration against Brownian, or $Q$%
-Wiener processes) is well-known to be consistent with Stratonovich
stochastic integration (e.g. \cite{LQ02, FV10, FH12}). In fact, one can also
construct a rough path lift via It\^{o}-integration, in this case (random)
rough integration (now against a "non-geometric" rough path) coincides with
It\^{o}-integration.

\section{An application}

Let $X$ be a continuous $d$-dimensional Gaussian process which admits a
rough path lift in the sense described at the end of the previous section.
Assume in addition that the Cameron-Martin space $\mathcal{H}$ has \emph{%
complementary Young regularity} in the sense that $\mathcal{H}$ embeds
continuously in $C^{q\text{-var}}\left( \left[ 0,T\right] ,\mathbb{R}%
^{d}\right) $ with $\frac{1}{p}+\frac{1}{q}>1$. Note $q\leq p$ for $\mu $ is
supported on the paths of finite $p$-variation. This is true in great
generality with $q=\rho $ whenever $\rho <3/2$ and also for fBM (and
variations thereof) for all $H>1/4$. Complementary Young regularity of the
Cameron-Martin space is a natural condition, in particular in the context of
Malliavin calculus and has been the basis of non-Markovian H\"{o}rmander
theory, the best results up to date were obtained in \cite{CF} (existence of
density only, no drift, general non-degenerate Gaussian driving noise) and
then \cite{HP11} (existence of a smooth density, with drift, fBM $H>1/3$).
We give a quick proof of existence of density, \textit{with drift,} with
general non-degenerate Gaussian driving noise (including fBM $H>1/4$). To
this end, consider the rough differential equation%
\begin{equation*}
dY=V_{0}\left( Y\right) dt+V\left( Y\right) d\mathbf{X}
\end{equation*}%
subject to a \textit{weak} H\"{o}rmander condition at the starting point.
(Vector fields, on $\mathbb{R}^{e}$, say are assumed to be bounded, with
bounded derivatives of all orders.) In the drift free case, $V_{0}=0$,
conditions on the Gaussian driving signal $X$ where given in \cite{CF} which
guarantee existence of a density. With no need of going into full detail
here, the proof (by contradiction) follows a classical pattern which
involves a deterministic, non-zero vector $z$ s.t. $z^{T}J_{0\leftarrow
\cdot }^{\mathbf{X}\left( \omega \right) }\left( V_{k}\left( Y_{\cdot
}\left( \omega \right) \right) \right) \equiv 0$ on $[0,\Theta \left( \omega
\right) ),$every $k\in \left\{ 1,\dots ,d\right\} $ for some a.s. positive
random time $\Theta $. (This follows from a global non-degeneracy condition,
which, for instance, rules out Brownian bridge type behaviour, and a 0-1
law, see conditions 3,4 in \cite{CF}). From this%
\begin{equation*}
\int_{0}^{\cdot }z^{T}J_{0\leftarrow t}^{\mathbf{X}}\left( \left[ V,V_{k}%
\right] \left( Y_{t}\right) \right) d\mathbf{X}+\int_{0}^{\cdot
}z^{T}J_{0\leftarrow t}^{\mathbf{X}}\left( \left[ V_{0},V_{k}\right] \left(
Y_{t}\right) \right) dt\equiv 0
\end{equation*}%
on $[0,\Theta \left( \omega \right) );$ here $V=\left( V_{1},\dots
V_{d}\right) $ and $V_{0}$ denote smooth vector fields on $\mathbb{R}^{e}$
along which the RDEs under consideration do not explode. Now we assume the
driving (rough) path to be truely rough, at least on a positive
neighbourhood of $0$. Since $Z:=\left( X,Y,J\right) $ can be constructed
simultanously as rough path, say $\mathbf{Z}$, we conclude with Theorem \ref%
{ThmDBRP}, (iii):%
\begin{equation*}
z^{T}J_{0\leftarrow \cdot }^{\mathbf{X}}\left( \left[ V_{l},V_{k}\right]
\left( Y_{\cdot }\right) \right) \equiv 0\equiv z^{T}J_{0\leftarrow \cdot }^{%
\mathbf{X}}\left( \left[ V_{0},V_{k}\right] \left( Y_{t}\right) \right) .
\end{equation*}%
Usual iteration of this argument shows that $z$ is orthogonal to $%
V_{1},\dots ,V_{d}$ and then all Lie-brackets (also allowing $V_{0}$),
always at $y_{0}$. Since the weak-H\"{o}rmander condition asserts precisely
that all these vector fields span the tangent space (at starting point $%
y_{0} $) we then find $z=0$ which is the desired contradiction. We note that
the true roughness condition on the driving (rough) path replaces the
support type condition put forward in \cite{CF}. Let us also note that this
argument allows a painfree handling of a drift vector field (not including
in \cite{CF}); examples include immediately fBM with $H>1/4$ but we have
explained above that far more general driving signals can be treated. In
fact, it transpires true roughness of $Q$-Wiener processes (and then,
suitables generalizations to $Q$-Gaussian processes) on a seperable Hilbert
space $\mathbb{H}$ allows to obtain a H\"{o}rmander type result where the $Q$%
-process "drives" countably many vectorfields given by $V:$ $\mathbb{R}%
^{e}\rightarrow Lin\left( \mathbb{H},\mathbb{R}^{e}\right) .$

The Norris type lemma put forward in \cite{HP11} suggests that that the
argument can be made quantitative, at least in finite dimensions, thus
allowing for a H\"{o}rmander type theory (existence of smooth densities) for
RDE driven by general non-degenerate Gaussian signals. (In \cite{HP11} the
authors obtain this result for fBM, $H>1/3\,$.)

\begin{acknowledgement}
P.K. Friz has received funding from the European Research Council under the
European Union's Seventh Framework Programme (FP7/2007-2013) / ERC grant
agreement nr. 258237. A. Shekhar is supported by Berlin Mathematical School
(BMS).
\end{acknowledgement}

\end{document}